\documentclass[11pt]{article}
\pdfoutput=1
\usepackage{amssymb}
\usepackage{amsmath}
\usepackage[top=1.2in, bottom=1.4in, left=1in, right=1in]{geometry}
\usepackage{parskip}
 \usepackage[T1]{fontenc}
\usepackage{lmodern}
 \usepackage{centernot}
\usepackage{epsfig}
\usepackage{xcolor}
\usepackage{amsmath,amsthm,amsfonts,amssymb,cite,amscd}
\usepackage{indentfirst}
\usepackage{mathrsfs}
\usepackage{amsthm}
\usepackage{tikz}
\usepackage{mathtools}
\usepackage{calc}
\usepackage[symbol]{footmisc}
\usetikzlibrary{patterns,arrows,decorations.pathreplacing}
\usetikzlibrary{fadings}
\usepackage{pgfplots}
\usepackage{parskip}
\usepackage{bbm}
\newtheorem{theorem}{Theorem}

\newtheorem{lemma}[theorem]{Lemma}

\newtheorem*{definition*}{Definition}
\newtheorem{claim}{Claim}
\newtheorem{conjecture}{Conjecture}

\newcommand{\floor}[1]{\left\lfloor{#1}\right\rfloor}
\newcommand{\ceil}[1]{\left\lceil{#1}\right\rceil}

\newcommand{\RNum}[1]{\uppercase\expandafter{\romannumeral #1\relax}}
\usepackage{authblk}
\title{Stability version of Dirac's theorem and its applications for generalized Tur\'an problems}
\author[1,2]{Xiutao Zhu}
\author[1]{Ervin Győri}
\author[1,3]{Zhen He}
\author[1,3]{Zequn Lv}
\author[1]{Nika Salia}
\author[1,4]{Chuanqi Xiao}

\affil[1]{Alfr\'ed R\'enyi Institute of Mathematics, Hungarian Academy of Sciences. }
\affil[2]{Department of Mathematics, Nanjing University. }
\affil[3]{Department of Mathematical Sciences, Tsinghua University.}
\affil[4]{Central European University, Budapest.}
\date{}

\DeclareMathOperator{\ex}{ex}

\def\C{\mathcal C}

\def\F{\mathcal F}

\newcommand{\abs}[1]{\left\vert{#1}\right\vert}

\begin{document}
\maketitle 
\begin{abstract}
 In 1952, Dirac  proved that every $2$-connected $n$-vertex graph with the minimum degree $k+1$ contains a cycle of length at least $\min\{n, 2(k+1)\}$.  Here we obtain a stability version of this result by characterizing those graphs with minimum degree $k$ and circumference at most $2k+1$.

 We present applications of the above-stated result by obtaining  generalized Tur\'an numbers.  
 In particular, for all $\ell \geq 5$ we determine how many copies of a five-cycle as well as four-cycle are necessary to guarantee that the graph has  circumference larger than $\ell$. 
 In addition, we give a new proof of Luo's  Theorem for cliques using our stability result. 

\end{abstract}
\section{Introduction}

\subsection*{Circumference of graphs}
The problem of determining whether a graph contains a Hamiltonian cycle has been a fundamental question of graph theory. 
Deciding the Hamiltonicity for graphs is NP-complete.  
Therefore it is interesting to study sufficient conditions for Hamiltonicity.
The natural generalization of this problem is to find sufficient conditions for a given circumference which is the length of a longest cycle. 
In 1952 Dirac obtained a bound on the circumference of $2$-connected  graphs in terms of the minimum degree. 
Let us denote the circumference of a graph $G$ by $c(G)$.

\begin{theorem}(Dirac\cite{dirac1952some})\label{thm:Dirac}
Let $G$ be a $2$-connected $n$-vertex graph with minimum degree at least $k+1$,  then 
\[
c(G)\ge \min\{n, 2(k+1)\}.
\]
\end{theorem}

Later in 1977, Kopylov obtained a similar bound on the circumference of $2$-connected  graphs in terms of the average degree. Let us denote the number of edges of a graph $G$ by $e(G)$.

\begin{theorem}(Kopylov\cite{kopylov1977maximal})\label{Thm:Kopylov}
Let $G$ be a $2$-connected $n$-vertex graph  with~$c(G) \leq \ell$ then 
\[
e(G)\leq  max\left\{\binom{\ell-1}{2}+2(n-\ell+1), \binom{\floor{\frac{\ell}{2}}}{2}+\floor{\frac{\ell}{2}} \left(n-\floor{\frac{\ell}{2}}\right)+\mathbbm{1}_{2|(\ell-1)}\right\}.
\]

\end{theorem} 

F\"{u}redi, Kostochka and Verstra\"{e}te \cite{furedi2016stability} obtained the stability result of Kopylov's Theorem. Before presenting the result we need to introduce a class of extremal graphs. 
Let $K_k$ be the  clique of $k$ vertices and $I_k$ be the independent set of $k$ vertices. For a positive integer $a$, let $aK_k$ be the graph consisting of~$a$ disjoint cliques of order $k$.
For graphs $G$ and $H$, we denote by $G \cup H$ the disjoint union of graphs $G$ and $H$. 
We denote by $G+H$ the join of $G$ and $H$, that is the graph obtained by connecting each pair of vertices between a vertex disjoint copies of $G$ and $H$. 
For example $K_k+I_{n-k}$  has minimum degree $k$  and circumference is $2k$ for $n\geq 2k$. 
For a set of vertices $A\subseteq V(G)$, let $G-A$ be the induced subgraph of $G$ on the vertex set $V(G)\setminus A$, i.e. $G-A=G[V(G)\setminus A]$.

\noindent \textbf{Introduction  of some classes of extremal graphs.}  We denote the graph $K_k+I_{n-k}$ by $H(n,2k)$ and let $H(n,2k+1)$ be a graph obtained from $H(n,2k)$ by adding an additional edge incident to two vertices of the independent set $I_{n-k}$. 


Here we define a class of graphs $\mathcal{H}_{1,n,k}$ for all integers $k$ and $n$ such that $n=b(k-1)+3$ for some positive integer $b$. Let~$b=b_1+b_2$ for some non-negative integers $b_1$ and $b_2$. Then let $G_0$ be the graph $\left((b_1K_{k-1}+\{u_1\})\cup(b_2K_{k-1}+\{u_2\})\right)+\{u\}$. Let $G$ be the graph obtained from $G_0$ by adding the edge $u_1u_2$, $G_1$ be the graph obtained from $G$ by removing the edge $uu_1$, $G_2$ be the graph obtained from $G$ by removing the edge $uu_1$ and $G_3$ be the graph obtained from $G$ by removing edges $uu_1$ and $uu_2$. All such graphs $G,G_1,G_2$ and $G_3$ are from the class $\mathcal{H}_{1,n,k}$. Note that all graphs in $\mathcal{H}_{1,n,k}$ have circumference $2k+1$. 


For all integers $k$ and $n$ such that $n=b(k-1)+1$ for some positive integer $b$, let 
\[
\mathcal{H}_{2,n,k}=\{K_2+bK_{k-1}, \overline{K}_2+bK_{k-1}\}.
\] 
Note that, the graphs from $\mathcal{H}_{2,n,k}$ have circumference $2k$.

\begin{theorem}( F\"{u}redi, Kostochka, Verstra\"{e}te\cite{furedi2016stability})
Let $G$ be a $2$-connected $n$-vertex graph such that~$c(G)=\ell$ and $n\ge 3\floor{\ell/2}$, then eitehr  
\[e(G)< \binom{\ceil{\ell/2}+2}{2}+\left(\floor{\frac{\ell}{2}}-1\right)\left(n-\ceil{\frac{\ell}{2}}-2\right),
\]
or~$G\subseteq H(n,\ell)$ or $G-A$ is a star forest for some $A\subseteq V(G)$ of size at most $\frac{\ell}{2}$.
\end{theorem} 

Recently, Ma and Ning also obtained more general stability-type
results of Kopylov’s Theorem in \cite{Majie}.
In this work we prove the following stability version of Dirac's theorem.

\begin{theorem}\label{thm:Main_Dirac}
Let $G$ be a $2$-connected graph of $n$ vertices with $n\ge c(G)+1$ and $\delta(G)= k$. Then either $c(G)\geq 2k+2$, or
\begin{itemize}
  \item $c(G)=2k+1$ and~$G\subseteq H(n,2k+1)$, or $G\in \mathcal{H}_{1,n,k}$, or $G\subseteq K_2+(K_k\cup \frac{n-k-2}{k-1}K_{k-1})$, or $k=4$ and $G\subseteq K_3+\frac{n-3}{2} K_2$, or $k=3$ and $G\subseteq K_2+(S_{n-3-2t}\cup tK_2)$.
  \item~$c(G)= 2k$ and~$G\subseteq H(n,2k)$ or $G\in  \mathcal{H}_{2,n,k}$.
  \end{itemize}
\end{theorem}

This theorem seems to have many applications. With this new tool, it is possible to 
re-prove some classical results in graph theory. Even more with this theorem we determined generalized Tur\'an numbers of cycles.

\subsection*{Applications for Generalised Tur\'an numbers.}

A central topic of extremal combinatorics  is to investigate sufficient conditions for the appearance of a given cycle. In particular, it is popular to maximize the number of cycles of length~$\ell$ in graphs of given order without a cycle of length $k$ as a subgraph. 
 For given integers~$k>3$ and~$m$, Gishboliner  and  Shapira determined the order of magnitude of how many copies of $k$-cycle is enough to guarantee the appearance of a $m$-cycle. This problem was also settled independently in  \cite{gerbner2017generalized} for $k$ and~$m$ even. Maximizing the number of triangles in $k$-cycle free graphs is still not settled, since this number is closely related to Tur\'an number of even cycles see \cite{gyHori2012maximum}. 

While Erd\H{o}s was measuring how far are the triangle-free graphs from bipartite graphs, he naturally asked a question 
\textit{`What is the maximum number of pentagons in a triangle-free graph'~\cite{erdos1984some}}. This question was settled half a century later by  Grzesik~\cite{grzesik2012maximum}  and independently by Hatami, Hladk\'y, Kr\'al, Norine, Razborov~\cite{hatami2013number}, using flag algebras.  In 1991, Győri, Pach, Simonovits~\cite{gyori1991maximal}, defined the generalized Tur\'an number and obtained some results. In particular, they maximized copies of a bipartite graph with an almost one-factor in triangle-free graphs. While investigating pentagon-free $3$-uniform hypergraphs  Bollob\'as-Gy\H ori~\cite{bollobas2008pentagons} initiated the study of the converse of the problem of Erd\H{o}s. They asked the following question `What is the maximum number of triangles in a pentagon-free graph'. This problem is still open, for the improvements on the upper-bound see \cite{bollobas2008pentagons, ergemlidze2019note, ergemlidze2018triangles}.

Grzesik and  Kielak in   \cite{grzesik2022maximum} determined that every graph on $n$ vertices
without odd cycles of length less than $k$ contains at most $(n/k)^k$ cycles of length $k$ for all $k\geq 7$. This result is an extension of the previously mentioned problem of Erd\H{o}s~\cite{erdos1984some}. Erd{\H{o}}s and Gallai determined the maximum number of edges in a graph not containing long paths and cycles as well in \cite{erdHos1959maximal}. Luo\cite{luo2018maximum} extended this result by determining the maximum number of cliques in a graph with a given circumference. The generalized Tur\'an version of this problem for paths was studied in \cite{plpk}.


\noindent \textbf{Notations.}
The cycle of length $\ell$ is denoted by~$C_\ell$.~$\C_{\ge \ell}$ denotes the family of all cycles of length at least~$\ell$.  For an integer $n$, a graph~$H$ and a family of graphs $\F$, Alon and Shikhelman denoted  generalized Tur\'an number by $\ex(n,H,\F)$ in \cite{alon2016many, alon2018additive}. Where~$\ex(n,H,\F)$ denotes the maximum number of copies of $H$ as a subgraph in an $n$-vertex graph not containing $F$ as a subgraph for all $F\in \F$. When family $\F$ consists of a single graph $F$, i.e. $\F=\{F\}$  we write $\ex(n,H,F)$ instead of $\ex(n,H,\{F\})$. 

 For  graphs $G$ and $H$ let $H(G)$ be the number of copies of $H$ in $G$. For example the number of cycles of length $k$ in $G$ is denoted by $C_k(G)$. For a vertex $v$ in a graph~$G$, let $C_t(v)$ be the number of cycles of length $t$ containing the vertex $v$ in $G$. For $v\in V(G)$, we denote  the neighborhood of $v$ by $N(v)$. For a vertex $v$, the closed neighbourhood of it~$N(v) \cup \{v\}$ is denoted by~$N[v]$. 

\noindent \textbf{Generalized Tur\'an-type results.}
 In this paper, by applying Theorem \ref{thm:Main_Dirac}, we determine the maximum number of four-cycles and pentagons in graphs with bounded circumference. Even more we prove that the extremal graph is unique for large enough $n$. 
 
\begin{theorem}\label{thm:c_5}
For all integers  $\ell \geq 6$ and $n\ge 100 \ell^{3/2}$ we  have
\[
\ex(n, C_5, \C_{\ge \ell+1})=C_5(H(n,\ell)),
\]
and $H(n,\ell)$ is the unique extremal graph. 

For $\ell=5$ and $n\ge 200$, we  have 
\[
\ex(n, C_5, \C_{\ge \ell+1})=\floor{\frac{(n-3)^2}{2}},
\]
the extremal graph is  a member of the family $\mathcal{H}_{1,n,k}$ with parameters $\floor{\frac{n-3}{2}}, \ceil{\frac{n-3}{2}}$.
\end{theorem}

\begin{theorem}\label{thm:c_4}
For all integers $n$ and $\ell$ such that $\ell \geq 4$ and $n\ge 10 \ell^{3/2}$, we have
\[
\ex(n, C_4, \C_{\ge \ell+1})=C_4(H(n,\ell)),
\]
and $H(n,\ell)$ is the unique extremal graph.
\end{theorem}

In addition, we also give a  new proof of Luo's following theorem by using Theorem \ref{thm:Main_Dirac}.
\begin{theorem}(Luo \cite{luo2018maximum})\label{Thm:Luo}
For all integers $n$ and $\ell \ge 3$ we have
\[
ex(n,K_s,\C_{\ge \ell+1})\le \frac{n-1}{\ell-1}\binom{\ell}{s}.
\]
The equality holds if and only if $\ell-1|n-1$.
\end{theorem}

We expect Theorem \ref{thm:c_5} holds not only for  cycles of length  four and five  but for cycles of any length more than 3.
\begin{conjecture}\label{Conjecture:general}
For all integers $n$, $k$ and $\ell$ such that $k\geq 4$, $\ell > k$ and $n$ large enough, we have
\[
\ex(n, C_k, \C_{\ge \ell+1})=C_{k}(H(n,\ell)).
\]
\end{conjecture}

We also prove the following theorem which verifies  Conjecture \ref{Conjecture:general} asymptotically for large enough $k$ and $n$. 
\begin{theorem}\label{thm:longer_cycles}
The following holds for every integer $k\geq 3 $.
\[  \displaystyle
 \lim_{\ell \to \infty} \left(\lim_{n \to \infty} \frac{\ex(n,C_{2k},\C_{\ge \ell+1})}{\floor{\frac{\ell}{2}}^{k} n^{k}}\right)= \frac{1}{2k},
\]

\[  \displaystyle
 \lim_{\ell \to \infty} \left(\lim_{n \to \infty} \frac{\ex(n,C_{2k+1},\C_{\ge \ell+1})}{\floor{\frac{\ell}{2}}^{k+1} n^{k}}\right)=\frac{1}{2}.
\]
\end{theorem}


\section{Preliminaries}

Erd\H{o}s and Gallai used the following robust lemma to find the extremal number of graphs with bounded circumference. We use the lemma to prove Theorem \ref{thm:Main_Dirac}. 
\begin{lemma}(Erd\H{o}s-Gallai\cite{erdHos1959maximal})\label{lemmma:Erdos_Gallai_min_degree_k_path}
Let $G$ be a $2$-connected graph and $x,y$ be two given vertices. If every vertex other than $x,y$ has a degree at least $k$ in $G$, then there is an $(x,y)$-path of length at least~$k$.
\end{lemma}

Even more, Li and Ning applied this lemma to prove the existence of $(H,C,t)$-fans under some conditions. For our proof of Theorem \ref{thm:Main_Dirac} we need the existence of $(H,C,t)$-brooms under the same conditions. Let us introduce the notion of $(H,C,t)$-brooms.

\begin{definition*}
Let $G$ be a graph, $C$ be a cycle of $G$, and $H$ be a component of $G-C$. A subgraph $B$ of $G$ is called an $(H,C,t)$-\textit{broom}, if it consists of~$t$ paths $P_1,P_2,\cdots,P_t$ each starting at the same vertex of $H$ and finishing at distinct vertices of $C$ for some~$t\ge 2$, such that\\
(1) All vertices of $P_1$ except the last are in $V(H)$.\\
(2) The paths $P_i$ have length one for all $2\leq i \leq t$.
\end{definition*}

The same proof of Theorem 2.1  in the paper of  Li and Ning \cite{li2021strengthening} proves the following theorem. Naturally, to refrain from repetition, we will not include their proof in this work.

\begin{lemma}(Li-Ning\cite{li2021strengthening})\label{lemma:H_C_t_broom}
Let $G$ be a $2$-connected graph, $C$ a cycle of $G$, and $H$ a connected component of $G-C$. If each vertex $v\in V(H)$ has $d_G(v)\ge k$, then there is an $(H,C,t)$-broom with at least $k$ edges.
\end{lemma}

\section{Proof of the Stability of Dirac's theorem}
Here we present the proof of Theorem \ref{thm:Main_Dirac}.
Let $G$ be an $n$-vertex $2$-connected graph with minimum degree $k\geq 2$ and circumference at most $2k+1$. By Theorem \ref{thm:Dirac}, $G$ contains a cycle of length at least $\min\{n,2k\}$.  Since $n\ge c(G)+1$, hence we have 
$c(G)\in \{2k+1,2k\}$. 
 Let $C$ be a longest cycle  of $G$ and $H_1,H_2,\cdots,H_s$ be connected components of $G-C$ for some $s\geq 1$, where $G-C$ is the induced subgraph of $G$ on the vertex set $V(G)\setminus V(C)$.  Since $\delta(G)\ge k$,  each component $H_j$ contains an $(H_j,C,t)$-broom with at least $k$ edges by Lemma \ref{lemma:H_C_t_broom}. In the following part of the proof we characterize the structure of each $H_i$. 

Let $B$ be an edge-maximal $(H_1,C,t)$-broom consisting of following $t$ paths $vPu_1, vu_2, \dots,$ and~$vu_t$. Recall that  vertices $\{u_1,u_2,\dots,u_t\}$ are distinct vertices of the cycle $C$. Each cycle has a positive and a negative direction to visit their vertices, without loss of generality we assume that starting at the vertex $u_1$ going around $C$ in the positive direction we visit terminal vertices of $B$ in this given order~$u_2,u_3\cdots u_t$.  For a given vertex $u$ of $C$, we denote its two neighbors on the cycle by $u^+$ and $u^-$, where~$u^-uu^+$ is a sub-path of $C$ in the positive direction.  For two vertices $x,y$ of $C$, $x\overrightarrow{C}y$ denotes the segment of $C$ from $x$ to $y$ in the positive direction, $x\overleftarrow{C}y$ denotes the  segment of $C$ from $x$ to $y$ in the negative direction. 

Recall the length of $C$ is $2k$ or $2k+1$ and $v(B)\geq k+1$ by Lemma \ref{lemma:H_C_t_broom}. On the other hand we have
\[
v(C)= t+ \sum_{i=1}^{t} v(u_i^+\overrightarrow{C}u_{i+1}^-)\geq t+ (t-2)+ 2(v(B)-t)=2v(B)-2,
\]
where indices are taken modulo $t$. Hence we have $v(B)=k+1$. Furthermore if $v(C)=2k$, then the segments $u_1^+Cu_2^-$ and $u_t^+Cu_1^-$ contain exactly $v(vPu_1)-1=k-t+1$ vertices while the rest of the segments contain exactly one vertex. If $v(c)=2k+1$, then one of the segments contains one more vertex.

\begin{claim}\label{Claim:H_1_structure}
  We have either~$k=4, c(G)=9$ and $G\subseteq K_3+\frac{n-3}{2} K_2$, or $t\in \{2,k\}$, $H_1=K_{k-t+1}$ and each vertex of $H_1$ is incident with all vertices of $\{u_1,\ldots,u_t\}$.  

\end{claim}

\begin{proof}
At first we assume $t=k$. Therefore all paths of the broom $B$ are  single edges and all segments $u_i^+\overrightarrow{C}u_{i+1}^-$ of cycle $C$ contain exactly one vertex except if $c(G)=2k$ and one segment containing two vertices if $c(G)=2k+1$. Without loss of generality, suppose $u_1^+\overrightarrow{C}u_2^-$ contains two vertices. 
We have $V(H_1)=\{v\}$ since otherwise we could extend the cycle $C$ given that $G$ is $C \cup H_1$ is $2$-connected. Hence we are done if $k=t$.


From here we assume~$2\leq t \le k-1$. The path $vPu_1$ is a path of~$k-t+2$ vertices, let~$vPu_1$ be $v_0v_1\cdots v_{k-t}u_1$, where $v_0=v$. 

First we assume $V(H_1)\not=\{v,v_1,\ldots,v_{k-t}\}$.  Let $H_1'$ be a maximal connected component of $H_1-\{v_0,v_1,\ldots,v_{k-t}\}$. Since $G$ is $2$-connected there are at least two edges from $H_1'$ to the rest of the graph. At first we suppose that  there is a vertex $y$ in $V(H_1')$ with a neighbour $u'$ on $C$. Since $H_1'$ is a subgraph of connected graph $H_1$, there is an edge $v_ix$ between $V(H_1')$ and $\{v_0,v_1,\ldots,v_{k-t}\}$ for some $i$ satisfying $0\le i\le k-t$. Since~$x$ and $y$ are vertices of $V(H_1')$ there exists a path $xP'y$  from $x$ to $y$ in $H_1'$. 
If $u'$ is on the segment $u_j^+\overrightarrow{C}u_{j+1}^-$ for some $j$ satisfying $1\le j\le t$, then 
\[
u_jPv_ixP'yu'\overrightarrow{C}u_j 
\text{~or~}
u_j\overrightarrow{C}u'yP'xv_iPv_0u_2Cu_j
\] 
is a longer cycle. Since otherwise $v(u_j^+\overrightarrow{C}(u')^-)\geq k-t-i+2$ and $v((u')^+\overrightarrow{C}u_{j+1}^-)\geq i+2$ contradicting to~$k-t+2\geq v(u_j^+\overrightarrow{C}u_{j+1}^-)=v(u_j^+\overrightarrow{C}(u')^-) +1+v((u')^+\overrightarrow{C}u_{j+1}^-)\geq k-t+5$. 
Moreover, if $u'=u_j$ for some $j$ satisfying~$3\le j\le t$, then
\[
u_{j-1}v_0Pv_ixP'yu_j{\overrightarrow{C}}u_{j-1}
\] 
is a longer cycle, a contradiction. 
Hence, $u'\in\{u_1,u_2\}$ and $N_{G}(V(H_1'))\subseteq \{v_0,v_1,\dots,v_{k-t},u_1,u_2\}$. 
Furthermore, if $t\ge 3$, $u'\not=u_2$ and $N_{G}(V(H_1'))\subseteq \{v_0,v_1,\dots,v_{k-t},u_1\}$. For otherwise, 
\[
u_3v_0Pv_ixP'yu_2\overleftarrow{C}u_3~~ or~~ u_1Pv_ixP'yu_2\overrightarrow{C}u_1
\]
is a longer cycle, a contradiction. (If it is not the first case, then $i=0$ and $v(u_1^+\overrightarrow{C}u_2^-)=k-t+1$)

Observe that no two consecutive vertices of the path $u_2vPu_1$ are incident to a vertex of $V(H_1')$. By the minimum degree condition, we have~$|V(H_1')|>1$ since $k\geq 3$. Since $G$ is $2$-connected, there are at least two independent edges between $\{v_0,\cdots,v_{k-t},u_1,u_2\}$ and $V(H_1')$. Note that $u_2,v_0,\cdots,v_{k-t},u_1$ is a path, for the technical reasons we denote $v_{-1}:=u_2$ and $v_{k-t+1}:=u_1$. 
From all such pairs of edges, we choose two independent edges $x_1v_i$ and $x_2v_j$ minimizing $j-i$ if $H_1'$ is $2$-connected. Otherwise we still minimize $j-i$ such that that $x_1$ is in one of the $2$-connected blocks of $H_1'$ containing exactly one cut vertex $x'$ of $H_1'$ denote by $B_1'$. 
The vertex $x_2$ is in any other $2$-connected blocks  of $H_1'$.  
From minimality of $j-i$, vertices  of $V(B_1'\setminus\{x'\})$ 
are not incident with vertices from~$\{v_{i+1},\cdots,v_{j-1}\}$. 
Every vertex of~$B_1'\setminus\{x'\}$ has degree at least $k$ in $G$. 
On the other hand they are incident with vertices from $V(B_1')$ and 
$\{v_{-1},\cdots,v_{k-t},v_{k-t+1}\} \setminus  \{v_{i+1},\cdots,v_{j-1}\}$.
Hence we have the degree of vertices $B_1'\setminus\{x'\}$  in $B_1'$ is  at least 
\[
k-\left\lceil\frac{i+2}{2}\right\rceil-\left\lceil\frac{k-t-j+2}{2}\right\rceil\ge j-i-1.
\]
Note that at least one of the vertices of $\{v_i,v_j\}$ is not from $\{u_1,u_2\}$, since $H_1'$ is subgraph of connected $H_1$. 
By Lemma \ref{lemmma:Erdos_Gallai_min_degree_k_path}, there is a path $x_1P_1x'$ in  the block $B_1'$ of length  at least $j-i-1$. Therefore there is a path $x_1P''x_2$ of length at least $j-i+1$ in $H_1'$, a contradiction to the maximality of the broom B. Since by exchanging $v_iPv_j$ with $v_i x_1P''x_2v_j$,  we would get a bigger broom. Therefore we have $V(H_1)=\{v_0,v_1,\ldots,v_{k-t}\}$

Here we show $N(v_i)\subseteq V(H_1)\cup \{u_1,\cdots,u_t\}$ for $0\le i\le k-t$. The statement holds for $v_0$, suppose some $v_i$ is adjacent to a vertex $u'$ which is on some segment $u_j^+\overrightarrow{C}u_{j+1}^-$. Then one of the following cycles is longer than $C$
\[
u'\overrightarrow{C}u_j Pv_i u' 
\text{~or~}
u_{j+1}\overrightarrow{C}u'v_iPv_0 u_{j+1},
\] 
a contradiction. Hence we have $N(v_i)\subseteq V(H_1)\cup \{u_1,\cdots,u_t\}$ for $0\le i\le k-t$.
From the minimum degree condition we have $k\leq d_G(v_i)\leq (v(H_1)-1)+t= k$. Hence 
$H_1$ is a clique and each vertex of $H_1$ is incident with all vertices in $\{u_1,\cdots,u_t\}$. 

If $t=2$, we have~$H_1$ is a copy of $K_{k-1}$ and each vertex is adjacent to both $u_1,u_2$. Therefore we are done in this case. 

If $t=3$, then consider the following cycle 
\[
u_3\overrightarrow{C}u_2v_{k-t}Pv_0u_3.
\]
Since the length of it is not greater than $C$ and $v(u_2^+\overrightarrow{C}u_3^-)\leq 2$, we have $k-t=1$ and the segment $u_2^+\overrightarrow{C}u_3^-$ contains exactly two vertices(This means $c(G)=2k+1$). From here it is straightforward to check that $G\subseteq K_3+(\frac{n-3}{2}K_2)$.

If $k>t\ge 4$, one of segment $u_{2}^+\overrightarrow{C}u_{3}^-$ or $u_{3}^+\overrightarrow{C}u_{4}^-$ contains one vertex. Without loss of generality we may assume~$v(u_{2}^+\overrightarrow{C}u_{3}^-)=1$. Therefore the cycle $u_3\overrightarrow{C}u_2 v_{k-t}Pv_0 u_3$ is a longer cycle than $C$, a contradiction.
\end{proof}

From Claim \ref{Claim:H_1_structure} we have either~$G\subseteq K_3+(\frac{n-3}{2}K_2)$ and $k=4$ or $G$ contains a longest cycle $C$ and each connected component of~$G-C$ is either a vertex  and adjacent to $k$ vertices on $C$, or a clique of size ${k-1}$ and all vertices of the clique are adjacent  to the same two vertices of $C$. If $H_i$ is a $(k-1)$-clique, we call the two neighbors of $H_i$ lying on $C$ the attached point.

First consider  that each $H_i$ is a clique of size $k-1$  and let $w_i,w_i'$ denote the two attached points of $H_i$ for all $1\leq i \leq s$.   If $c(G)=2k$, one can easily check that~$v(w_i^+\overrightarrow{C}w_i'^-)=v(w_i'^+\overrightarrow{C}w_i^-)=k-1$, $\{w_i,w_i'\}=\{w_1,w_1'\}$ and $w_i^+\overrightarrow{C}w_i'^-, w_i'^+\overrightarrow{C}w_i^-$ are both a copy of $K_{k-1}$, hence $G\in \mathcal{H}_{2,n,k}$. 
When $c(G)=2k+1$, by Claim~\ref{Claim:H_1_structure}, we say the segment~$w_1\overrightarrow{C}w_1'$ contains $k$ vertices. If $s=1$, then $G\subseteq  K_2+(K_k\cup 2K_{k-1})$. If $s\ge 2$, then since $c(G)=2k+1$,  we have $\{ w_1,w_1'\}\cap \{ w_i,w_i'\}\neq \emptyset$ for any $2\le i\le s$.   Therefore either all $H_i$ have the same two attached points $\{w_1,w_1'\}$ on $C$ and we can see the segment $w_1\overrightarrow{C}w_1'$ as a subgraph of $K_k$ and we obtain $G\subseteq  K_2+(K_k\cup \frac{n-k-2}{k-1}K_{k-1})$. Or there are two of them such that their neighbours on $C$ are $w_1,w_1'$ and $w_1,w_1'^-$ and $G\in \mathcal{H}_{1,n,k}$, this finishes the proof in this case.

Next consider the case there is a component of $G-C$ of size one. Let us denote this vertex by $v$. The vertex $v$ has $k$ neighbours on the cycle $C$ and set $N(v)=\{u_1,\ldots,u_k\}$. Even more the distance between any two consecutive neighbours of $v$ is exactly two if $c(G)=2k$ and with one has distance three if $c(G)=2k+1$(if in such case, we assume $u_1\overrightarrow{C}u_2$ is of distance 3). It is easy to see that for any other components of size $1$, they have the same neighborhood with $v$ since $C$ is the longest. First assume there is no other component of size $k-1$. 
If $c(G)=2k$, then $V(C)-N(v)$ is independent and hence $G\subseteq H(n,2k)$. If $c(G)=2k+1$, then $V(C)-N(v)$ contains exactly one edge which lies on the segment of distance $3$ between two consecutive neighbours of $v$. Hence $G\subseteq H(n,2k+1)$. 

Hence we may assume that some component are $(k-1)$-cliques with $k\ge 3$, saying $H_i$ is one of such component with two attach points $\{u',u''\}$. If one of the attached points of $H_i$ lies on $u_i^+\overrightarrow{C}u^-_{i+1}$, we will find a longer cycle using $H_i$, a contradiction. Thus $\{u',u''\}\subseteq N(v)$ and we set $u'=u_a,~u''=u_b$ with $a,b\in  [k]$. If $k\ge 4$, then by the distance of $u'\overrightarrow{C}u''$, we know $u_{b-1}\not=u_a$ and $u_{a+1}\not=u_{b}$. We have $u_aH_iu_b\overleftarrow{C}u_{a+1}vu_{b-1}\overrightarrow{C}u_a$ is a longer cycle, a contradiction. Then $k=3$ and it is easy to see that $c(G)=7$ and  $u_a=u_1,~u_b=u_2$. That is $G-\{u_1,u_2\}$ is the disjoint union of a star and matching, $G=K_2+(S_{n-3-2t}\cup tK_2)$.
This finishes the proof of Theorem \ref{thm:Main_Dirac}.
 $\hfill\blacksquare$

\section{The applications for
generalized Tur\'an problems}

In this chapter we present some applications of   Theorem \ref{thm:Main_Dirac}. In particular we determine the exact value of the generalized Tur\'an number of pentagons or $C_4$ in graphs with bounded circumference and give a new proof of Theorem \ref{Thm:Luo}.

\subsection*{Proof of Theorem \ref{thm:c_5}.}

Throughout this subsection we denote~$\lfloor\ell/2\rfloor:=k$ and $\lambda:=\ell-2k$.

\begin{lemma}\label{Lemma_C_5_extremal_numbers}
Let $F$ be a graph isomorphic to an $n$-vertex graph from the following set
\[
\left\{H(n,\ell),~K_2+(K_k\cup bK_{k-1}),~K_3+\frac{n-3}{2}K_2,K_2+(S_{n-3-2t}\cup tK_2)\right\}\cup \mathcal{H}_{1,n,k}\cup \mathcal{H}_{2,n,k}. 
\] 
We have 
\begin{itemize}
  \item If $\ell\ge 6$ and $n\ge 3k$,
\[C_5(F)\leq C_5(H(n,\ell)).
\]

The equality holds if and only if $F=H(n,\ell)$.
  \item~If $\ell=5$ and $n\ge 7$, then $F\in \mathcal{H}_{1,n,k}$ with parameters $\floor{\frac{n-3}{2}}$ and $\ceil{\frac{n-3}{2}}$ contains most $C_5$.
  \end{itemize}

\end{lemma}
\begin{proof}
It is straightforward to determine the number of five cycles in $H(n,\ell)$.
\begin{align}\label{euqtion:4}
C_5(H(n,\ell))=&\binom{n-k}{2}\binom{k}{3}\cdot 3\cdot2+(n-k)\binom{k}{4}\binom{4}{2}\cdot2+\binom{k}{5}\frac{5!}{10}\notag\\
&+\lambda\left\{(n-k-2)\binom{k}{2}\cdot2+\binom{k}{3}\cdot3\cdot2\right\}.
\end{align}
Suppose $F\in \mathcal{H}_{1,n,k}$, with parameters $b_1$ and $b_2$. If $\ell\ge 6$ and $n\ge 3k$, then the number of pentagons in $F$ is 
 \begin{align*}
C_5(F)&= \frac{n-3}{k-1}\binom{k+1}{5}\frac{5!}{10}+2\left(\binom{b_1}{2}+\binom{b_2}{2}\right)\binom{k-1}{2}(k-1)\\
&+2(b_1+b_2)\binom{k-1}{2} +b_1b_2(k-1)^2\\
&\le \frac{n-3}{k-1}\binom{k+1}{5}\frac{5!}{10}+2\binom{\frac{n-3}{k-1}}{2}\binom{k-1}{2}(k-1)<C_5(H(n,\ell)).
\end{align*}
If $\ell=5$, then $C_5(F)=b_1b_2$. It is easy to see when $b_1=\floor{\frac{n-3}{2}}$ and $b_2=\ceil{\frac{n-3}{2}}$, $C_5(F)$ attains maximum, which is greater than $C_5(H(n,5))=2(n-4)$.

If $F\in \mathcal{H}_{2,n,k}$,  then the number of pentagons in $F$ is 
\begin{align*}
C_5(F)=\frac{n-2}{k-1}\binom{k+1}{5}\frac{5!}{10}+2\binom{\frac{n-2}{k-1}}{2}\binom{k-1}{2}(k-1)< C_5(H(n,\ell)).
\end{align*}
If $F= K_2+(K_k\cup bK_{k-1})$  with parameters $b_1$ and $b_2$,  then the number of pentagons in $F$ is 
\begin{align*}
C_5(F)= &\left(b\binom{k+1}{5}+\binom{k+2}{5}\right)\frac{5!}{10}+2\binom{b}{2}\binom{k-1}{2}(k-1)+2\left(kb\binom{k-1}{2}+b\binom{k}{2}(k-1)\right)\\
< &C_5(H(n,\ell)).
\end{align*}
If $F=K_3+\frac{n-3}{2}K_2$,  then the number of pentagons in $F$ is 
\begin{align*}
C_5(F)=\frac{n-3}{2}\frac{5!}{10}+2\binom{\frac{n-3}{2}}{2}(2*3*2)+\binom{\frac{n-3}{2}}{2}(2*2*3*2)<C_5(H(n,\ell)).
\end{align*}
If $F=K_2+(S_{n-3-2t}\cup tK_2)$, then the number of pentagons in $F$ is
\[
2\binom{t}{2}*2*2+\binom{s}{2}(2+4)+4ts+2(s+1)t<C_5(H(n,\ell))
\]
\end{proof}

\begin{lemma}\label{lem3}
Let $G$ be a $2$-connected $\C_{\ge \ell+1}$-free graph with $n$ vertices, such that~$n\ge 3k$.  For a vertex $v$ of $G$ with degree $d(v)\le k-1$, we have 
\[
C_5(v)\le k(k-2)^2n-\frac{1}{2}k^2(k-2)^2.
\]
\end{lemma}
\begin{proof}
We denote  the set of vertices $V(G)-N[v]$ by $N_2(v)$. Let $e_1$ be the number of edges in $G[N(v)]$, $e_2$ be the number of edges between the sets of vertices $N(v)$ and $N_2(v)$ and $e_3$ be the number of edges in $G[N_2(v)]$ respectively. 
Since $G$ is a $2$-connected $\C_{\ge \ell+1}$-free  graph with $n$-vertices such that~$n\ge 3k$, by Theorem \ref{Thm:Kopylov} we have
\begin{align}\label{2}
e_1+e_2+e_3\le k(n-k)+\binom{k}{2}+\lambda.    
\end{align}

Here we classify pentagons $vv_1v_2v_3v_4v$ incident with the vertex $v$ in $G$. We say $vv_1v_2v_3v_4v$ is Type-$i$ if $i=\abs{\{v_2,v_3\}\cap N(v)}$.

In this paragraph, we estimate  the  maximum number of  Type-2 pentagons. There are at most $e_1$ choices representing an edge $v_2v_3$. After fixing such an edge, there are at most $\binom{\abs{N(v)}-2}{2}$ choices for the pair of vertices $v_1$ and $v_4$. Hence the number of Type-2 pentagons in $G$ is at most 
\[
e_1\binom{d(v)-2}{2}\cdot2\le e_1(k-3)(k-4).
\]

Here we estimate  the  maximum number of  Type-1 pentagons. Note that the opposite edge of $v$ in the pentagon must be between~$N(v)$ and $N_2(v)$. Hence there are at most $e_2$ choices for such an edge. After fixing such an edge there are $\binom{\abs{N(v)}-1}{2}$ choices for vertices $v_1$ and $v_2$. Hence the number of Type-1 pentagons in $G$ is at most 
\[
e_2\binom{d(v)-1}{2}\cdot2\le e_2(k-2)(k-3).
\]
Here we estimate  the  maximum number of  Type-0 pentagons.  If each vertex of $N_2(v)$ has at most $k-2$ neighbors in $N(v)$, then the   number of Type-0 pentagons in $G$ is at most 
$e_3(k-2)(k-2)$.
Therefore, by inequality~\eqref{2}, we have 
\begin{align*}
 C_5(v)&\le (e_1+e_2+e_3)(k-2)^2\le \left(k(n-k)+\binom{k}{2}+\lambda\right)(k-2)^2\\
&\le k(k-2)^2n-\frac{1}{2}k^2(k-2)^2.
\end{align*}

If there is a vertex of $N_2(v)$ with $k-1$ neighbors in $N(v)$, then we have $d(v)=k-1$. We partition~$N_2(v)$ into two sets $A$ and $B$. Such that $A$ contains  all vertices in $N_2(v)$ with at least $k-2$ neighbors in $N(v)$. The set of remaining vertices $N_2(v)\setminus A$ is denoted by $B$.  
Let $e_3'$ denote the number of edges in $G[A]$ and $e_3'':=e_3-e_3'$. In particular $e_3''$ denotes number of edges in~$N_2(v)$ incident with at least one vertex from $B$.
The number of Type-0 pentagons is at most
\[
e_3'(k-1)(k-2)+e_3''(k-3)(k-2)
\]
and 
\begin{align}\label{3}
C_5(v)\le e_3'(k-1)(k-2)+(e_1+e_2+e_3'')(k-2)(k-3).   
\end{align}

If $|A|\le k+1$, we have $e_3'\le \binom{k+1}{2}.$
By inequality~\eqref{2} and the  above inequality we have 
\begin{align*}
C_5(v)&\le \binom{k+1}{2}(k-1)(k-2)+k(n-k)(k-2)(k-3)\\
&= k(k-2)(k-3)n-\frac{k^4}{2}+4k^3-\frac{13k^2}{2}+k
\\
& \le k(k-2)^2n-\frac{1}{2}k^2(k-2)^2.
\end{align*}

If $|A|\ge k+2$ then we distinguish two cases for estimating $e_3'$ depending on the value of $\lambda$. If $\ell=2k$, then $G[A]$ is~$P_4$-free, $P_3\cup P_2$-free and $3P_2$-free since $G$ is $C_{\ell+1}$-free. 

This implies $e(G[A])=e_3'\le n-k-1$. If $\ell=2k+1$,  then~$G[A]$ is $P_5$-free, $P_4\cup P_2$-free, $P_3\cup2P_2$-free, $2P_3$-free and~$4P_2$-free. Which implies~$e_3'\le n-k$.
Hence by inequality~\eqref{2} and the inequality~\eqref{3}  we have
\begin{align*}
C_5(v)&\le(n-k)(k-1)(k-2)+\left(nk-\frac{k(k+1)}{2}-(n-k)+1\right)(k-2)(k-3)\\
&=k(k-2)(k-3)n+2(k-2)(n-k)-\frac{1}{2}(k+1)k(k-2)(k-3)+(k-2)(k-3)\\
&\le k(k-2)^2n-\frac{1}{2}k^2(k-2)^2.
\end{align*}
We are done. \end{proof}

Here we finish the proof of Theorem \ref{thm:c_5} by means of  progressive induction. Let $G_n$ denote  an extremal graph of $ex(n,C_5,\C_{\ge \ell+1})$. We may assume $G_n$ is connected.  First we prove the case $\ell\ge 6$.   Let us define the following function.
\[
\phi(n)=\ex(n,C_5, \C_{\ge \ell+1})-C_5(H(n,\ell)).
\]
Note that $\phi(n)=C_5(G_n)-C_5(H(n,\ell))$ and it is a non-negative integer. 
In the following claim we find an upper-bound for $\phi(n)$.
\begin{claim}\label{claim1}
 For all $n$ such that~$n\ge 100k$, either $G_n=H(n,\ell)$, or 
 \[
\phi(n) \leq  \phi(n-1)- k(k-2)(n-4k).
\]
\end{claim}

\begin{proof}
By the definition of $\phi$ we have
\begin{align*}
\phi(n-1)-\phi(n)&= \left( C_5(H(n,\ell))-C_5(H(n-1,\ell))\right)-\left( C_5(G_n)-C_5(G_{n-1})\right).
\end{align*}
Therefore from equality~\eqref{euqtion:4}, we get
\begin{align}\label{1}
   C_5(H(n,\ell))-C_5(H(n-1,\ell))&=k(k-1)(k-2)\left(n-\frac{k+5}{2}\right)+\lambda k(k-1). 
\end{align}

If $G_n$ contains a cut vertex, let $B_1$ and $B_2$ be two end-blocks of $G_n$ with $|V(B_2)|\ge |V(B_1)|$ and let $b_1,b_2$ be the cut vertices of $B_1$ and $B_2$, respectively. 
At first we  assume $V(B_2)\ge|V(B_1)|\ge 3k$ and $\delta(B_i)\ge k$ for each $i=1,2$. 
Since each $B_i$ is $2$-connected, combining Theorem~\ref{thm:Main_Dirac} with Lemma~\ref{Lemma_C_5_extremal_numbers},   
we have $B_i=H(|V(B_i)|,\ell)$. 
A contradiction to the maximality of the number of pentagons in $G_n$, since we have 
\[
C_5(H(v(B_1),\ell))+C_5(H(v(B_2),\ell))< C_5(H(v(B_1)-1,\ell))+C_5(H(v(B_2)+1,\ell)),
\]
by convexity. Note that we could exchange $B_1$ and $B_2$ with $H(v(B_1)-1,\ell)$ and $H(v(B_2)+1,\ell)$ since they are the end-blocks.  
Hence, either $v(B_1)\le 3k$ or $\delta(B_i)\le k-1$ for some $B_i$.
If $v(B_1)\le 3k$ then let $v$ be a vertex other than $b_1$ in $B_1$, then since $n\ge 100 k$,
$$C_5(v)\le 12\binom{3k}{4}\le k(k-2)^2(n-\frac{k}{2}).$$ 
This implies
\begin{align}\label{5}
C_5(G_n)-C_5(G_{n-1})\le C_5(v)\le  k(k-2)^2(n-\frac{k}{2}).    
\end{align}
For the latter case $\delta(B_i)\le k-1$ for some $B_i$
without loss of generality, assume there is a vertex $v$ in $B_1$ such that $v$ has at most $k-1$ neighbors in $B_1$. 
If $v\not= b_1$, then since $B_1$ is $2$-connected and $v(B_1)\ge 3k$,  inequality~\eqref{5} holds by Lemma \ref{lem3}.
If $v=b_1$, we remove all  edges incident to $b_1$ in the subgraph $B_1$. We destroyed at most $k(k-2)^2n-\frac{1}{2}k^2(k-2)^2$ copies of $C_5$ by Lemma \ref{lem3}. Even more the resulting graph is disconnected graph on $n$ vertices. Therefore it contains at most~$C_5(G_{n-1})$ pentagons, since we could identify a vertex from each connected component. Thus  the inequality~\eqref{5} holds in this case too.

Combining equality~\eqref{1} and inequality~\eqref{5}, we get
\begin{align*}
\phi(n-1)-\phi(n)&\ge k(k-1)(k-2)\left(n-\frac{k+5}{2}\right)+\lambda k(k-1)- k(k-2)^2(n-\frac{k}{2})\\ 
&\ge  k(k-2)(n-4k),
\end{align*}
therefore we are done if $G_n$ is not $2$-connected.

If $G_n$ is $2$-connected and it contains a vertex~$v$  of degree at most $k-1$, then by Lemma \ref{lem3} we have~$C_5(v)\le k(k-2)^2n-\frac{1}{2}k^2(k-2)^2$, hence~$\phi(n)-\phi(n-1)\ge k(k-2)(n-4k)$ holds and we are done. 
If $\delta(G)\ge k$, then combining Theorem \ref{thm:Main_Dirac} and Lemma \ref{Lemma_C_5_extremal_numbers}, we have $G_n=H(n,\ell)$. 
\end{proof}

The function $\phi(n)$  is decreasing non-negative function. 
We have a trivial bound 
\[
\phi(100k)\le \binom{100k}{5}\frac{5!}{10}-C_5(H(100k,\ell))\le 10^9k^5.
\]

For each $n$ such that $n>100k$ we have either~$C_5(H(n,\ell))=ex(n, C_5,\C_{\ge \ell+1})$ and $\phi(n)=0$ or 
$\phi(n)\not=0$ and we have 
\[
\phi(n)\le\phi(100k)-k(k-2)\sum_{i=100k}^{n}(i-4k)\le 10^9k^5-\frac{n+92k}{2}(n-100k),
\]
by Claim \ref{claim1}. Therefore for all $n\ge 10^5k^{3/2}$ we have~$\phi(n)= 0$. Hence we have
 $ex(n, C_5,\C_{\ge \ell+1})=C_5(H(n,\ell))$.

Next we prove the special case when $\ell=5$ using  progressive induction. Note that   $k=2$.  Let the graph from $\mathcal{H}_{1,n,k}$ with parameters $\floor{\frac{n-3}{2}}, \ceil{\frac{n-3}{2}}$ be denoted by $F_n$ and $ \phi(n)=\ex(n,C_5, \C_{\ge \ell+1})-C_5(F_n)$.
\begin{claim}\label{claim3}
 For all $n$ such that~$n\ge 29$, either $G_n=F_n$, or 
 \[
\phi(n) \leq  \phi(n-1)- \floor{\frac{n-27}{2}}.
\]
\end{claim}
\begin{proof}
 If the extremal graph $G_n$ is $2$-connected, then by  Theorem~\ref{thm:Main_Dirac} and Lemma~\ref{Lemma_C_5_extremal_numbers} we have $G_n=F_n$ and we are done.

If $G_n$ is not $2$-connected then let $B_1,~B_2$ be two distinct end-blocks of $G_n$  such that $v(B_2)\ge v(B_1)$. If $v(B_1)\le 5$, then by removing a vertex of degree at most four from $B_1$ we  destroy at most $12$ copies of $C_5$. Hence we have  
$\phi(n-1)-\phi(n)\geq C_5(F_n)-C_5(F_{n-1})-12=\floor{\frac{n-3}{2}}-12$. 

If $v(B_1),v(B_2)\ge 6$, then note that $\delta(B_1),\delta(B_2)\geq 2$, we have  $B_1=F_{v(B_1)}, B_2=F_{v(B_2)}$ by Theorem~\ref{thm:Main_Dirac} and Lemma~\ref{Lemma_C_5_extremal_numbers} or $B_1=H(6,5)$. By convexity  of the number of pentagons in $F_n$ and $H(n,5)$, $G_n$ is not the extremal graph, a  contradiction. 
\end{proof} 
 
By  Claim~\ref{claim3}, we start progressive induction from $n=29$ and when  $n\ge 200$, we get $G_n$ is $2$-connected and $G_n=F_n$. This completes the proof of Theorem \ref{thm:c_5}.
 $\hfill\blacksquare$

\subsection*{Proof of Theorem \ref{thm:c_4}}
The proof of Theorem~\ref{thm:c_4} is very similar to the proof of Theorem~\ref{thm:c_5}. At first we prove the following lemmas. 
\begin{lemma}\label{Lemma:C_4_number_of_C_4-in_extremal}
For all~$n\ge \ell$, among all graphs in the set $\{H(n,\ell), K_2+(K_k\cup bK_{k-1}), K_3+\frac{n-3}{2}K_2\}\cup \mathcal{H}_{1,n,k}\cup \mathcal{H}_{2,n,k}$, $H(n,\ell)$ contains most copies of $C_4$.
\end{lemma}
We omit the proof since the proof is straightforward and similar to Lemma~\ref{Lemma_C_5_extremal_numbers}

\begin{lemma}\label{Lemma:C_4_number_of_C_4_on_given_vertex}
Let $G$ be a $2$-connected $\C_{\ge \ell+1}$-free graph on $n$ vertices. If some vertex $v$ has degree at most $k-1$, then 
\[
C_4(v)\le \binom{k-1}{2}n.
\]
\end{lemma}
\begin{proof}
    The number of ways to choose adjacent vertices of $v$ in a $C_4$ is at most $\binom{k-1}{2}$ and the number of choices for the opposite vertex of $v$ is at most $n$, hence we have $C_4(v)\le \binom{k-1}{2}n$.
\end{proof}

To finish the proof we also use progressive induction method. Let us define the following function  
\[
\phi(n)=\ex(n,C_4, \C_{\ge \ell+1})-C_4(H(n,\ell)).
\]
Using the same technique as in Claim \ref{claim1}, we have either the extremal graph $G_n$ is $2$-connected with $\delta(G)\ge k$ hence $G_n=H(n,\ell)$, or 
\begin{align*}
 \phi(n-1)-\phi(n)&=C_4(H(n,\ell))-C_4(H(n-1,\ell))-C_4(v)\\
&\ge  \binom{k}{2}(n-k-1)+3\binom{k}{3}- \binom{k-1}{2}n\\
&\ge (k-1)n-\frac{3k(k-1)}{2}\ge (k-1)(n-2k).
\end{align*}
The function $\phi(n)$  is decreasing non-negative function. 
We have a trivial bound 
\[
\phi(4k)\le 3\binom{4k}{4}-(k-1)n\left(\frac{n-4k}{2}\right).
\]

Therefore for all~$n\ge 10k^{\frac{3}{2}}$ we have~$\phi(n)=0$.  Hence we have
 $ex(n, C_4,\C_{\ge \ell+1})=C_4(H(n,\ell)$, this completes the proof of Theorem \ref{thm:c_4}.
 $\hfill\blacksquare$

\subsection*{A new proof of Luo's Theorem.}

We prove Theorem~\ref{Thm:Luo} by induction on the number of vertices $n$. If $n\le \ell$, then the
theorem trivially holds. In what follows we  prove the theorem for $n\geq \ell+1$ assuming it holds for all graphs with smaller number of vertices. 

Note that we may assume that $G$ is connected, otherwise, we are done by induction on each component.  If $G$ is $2$-connected and $\delta(G)\ge \floor{\frac{\ell}{2}}$, then by Theorem \ref{thm:Main_Dirac} we have 
\[K_s(G)\le K_s(H(n,\ell))< \frac{n-1}{\ell-1}\binom{\ell}{s}.\]
If $G$ is $2$-connected and some vertex $v$ has degree less than $ \floor{\frac{\ell}{2}}$, then
\[
K_s(G)\le K_s(G-v)+\binom{\floor{\frac{\ell}{2}}-1}{s-1}<\frac{n-1}{\ell-1}\binom{\ell}{s},
\]
by induction hypothesis.

If $G$ is not $2$-connected, let $B_1$ be the $2$-connected end-block with the cut vertex $v$.  Then by the induction hypothesis we have
\begin{align*}
  K_s(G)=K_s(B_1)+K_s(G-(V(B_1)\setminus \{v\}))&\le \frac{v(B_1)-1}{\ell-1}\binom{\ell}{s}+\frac{(n-v(B_1)+1)-1}{\ell-1}\binom{\ell}{s} \\
  &=\frac{n-1}{\ell-1}\binom{\ell}{s}.
\end{align*}
Equality holds if and only if $\ell-1|n-1$ and each maximal 2-connected block is a copy of $K_\ell$.
 $\hfill\blacksquare$

\section{Counting general cycles}
In this section we prove Theorem~\ref{thm:longer_cycles}. At first note that $H(n,\ell)$ provides a  lower-bound for the  number of $C_{2k}$ and $C_{2k+1}$ as well.  

At first we will show  
\[ 
\displaystyle
 \lim_{\ell \to \infty} \left(\lim_{n \to \infty} \frac{\ex(n,C_{2k},\C_{\ge \ell+1})}{\floor{\frac{\ell}{2}}^{k} n^{k}}\right)\leq \frac{1}{2k}.
\]

Let $G$ be a $2$-connected graph with circumference at most $\ell$. Then by Theorem~\ref{Thm:Kopylov} we have $e(G)\leq \floor{\frac{\ell}{2}} n$. Let~$e_1,e_2,\dots, e_k$ be $k$ independent edges such that  there are no more than two cycles of length $2k$ containing edges $e_1,e_2,\dots, e_k$ in this given order. Then the number of $2k$-cycles on such $k$ independent edges in $G$ is at most
\[
2\frac{(\floor{\frac{\ell}{2}} n)^k}{4k}.
\]
Which is the desired upper bound in case the rest of the cycles are negligible. Indeed for  independent edges $e_1,e_2,\dots, e_k$  if there are more than two cycles of length $2k$ containing edges $e_1,e_2,\dots, e_k$ in this given order then  the induced graph on the vertex set $\cup_{i=1}^{k}\{v_i,u_i\}$ contains $2C_3\cup (k-3)P_2$ as a subgraph where $e_i=v_iu_i$. Hence the number of such cycles is at most
\[
(2k)! \left(\frac{2}{3} \floor{\frac{\ell}{2}}^{2} n \right)^2 (\floor{\frac{\ell}{2}} n)^{k-3}
\]
where we use Theorem~\ref{Thm:Luo} to bound the number of cycles and  Theorem~\ref{Thm:Kopylov} to bound the number of edges.  This shows the desired upper-bound.

We use induction on the number of vertices to show
\[  \displaystyle
 \lim_{\ell \to \infty} \left(\lim_{n \to \infty} \frac{\ex(n,C_{2k+1},\C_{\ge \ell+1})}{\floor{\frac{\ell}{2}}^{k+1} n^{k}}\right)\leq \frac{1}{2}.
\]
Observe that it is enough to show that there exists a vertex incident to at most 
\[
\frac{k}{2} \floor{\frac{\ell}{2}}^{k+1} n^{k-1}+ N_{\ell} k \ell^{k+1}n^{k-2}
\]
cycles of length $2k+1$, for some constant $N_{\ell}$.
By Dirac's theorem we have a vertex of $G$ with degree at most $\floor{\frac{\ell}{2}}$. Let $v$ be a vertex of minimum degree. Let us fix two vertices $w_1$ and $w_2$ adjacent to~$v$. 

\begin{claim}
  The number of paths of length $2k-1$ from $w_1$ to $w_2$ is at most
  \[
k \floor{\frac{\ell}{2}n}^{k-1} + N'_{\ell}\ell^{k-1} n^{k-2}
  \]
for some constant $N'_{\ell}$ dependent on $\ell$.   
\end{claim}
\begin{proof}
    The number of such $2k-1$-paths with terminal vertices $w_1$ and~$w_2$ with  a subgraph isomorphic to $K_4\cup (k-3) K_2$ or $2K_3 \cup (k-4) K_2$ is bounded by $N'_{\ell}\ell^{k-1} n^{k-2}$ by Theorem~\ref{Thm:Luo}, for some constant $N'_{\ell}$.

The number of $2k-1$-paths with terminal vertices $w_1$ and~$w_2$ using the fixed $k-1$ independent edges in the given order without having a subgraph $K_4\cup (k-3) K_2$ or $2K_3 \cup (k-4) K_2$ is at most~$k$. Hence we have the number of $2k-1$-paths with terminal vertices $w_1$ and~$w_2$ is at most 
 \[
k \floor{\frac{\ell}{2}}^{k-1}n^{k-1} + N'_{\ell}\ell^{k-1} n^{k-2}
  \]

\end{proof}
 
The number of cycles of length $2k+1$ incident with this vertex is at most
\[
\binom{\floor{\frac{\ell}{2}}}{2}k \left(\floor{\frac{\ell}{2}}^{k-1}n^{k-1} + N'_{\ell}\ell^{k-1} n^{k-2}\right).
\]
This finishes the proof. $\hfill\blacksquare$

\section{Acknowledgements}
We would like to thank Yixiao Zhang for useful remarks on the manuscript.
The research of Gy\H{o}ri and Salia was supported by the National Research, Development and Innovation Office NKFIH, grants  K132696 and SNN-135643. 

\bibliography{References.bib}

\textit{E-mail addresses:} \\
  X.~Zhu: \texttt{ zhuxt@smail.nju.edu.cn}\\
  E.~Gy\H{o}ri: \texttt{gyori.ervin@renyi.hu}\\
  Z.~He: \texttt{hz18@mails.tsinghua.edu.cn}\\
  J.~Lv: \texttt{lvzq19@mails.tsinghua.edu.cn}\\
  N.~Salia: \texttt{nikasalia@yahoo.com}\\
  C.~Xiao: \texttt{chuanqixm@gmail.com}\\

\end{document}